\documentclass[12pt]{article}
\usepackage[top=1.5in, bottom=1.5in, left=1.4in, right=1.4in]{geometry} 
\usepackage{amssymb}
\usepackage{graphicx}

\usepackage{color}
\usepackage{latexsym}
\usepackage{amssymb,amsmath,amstext}
\usepackage{epsfig}
\usepackage{graphics}
\usepackage{subfigure}
\usepackage{euscript}
\usepackage{ifthen}
\usepackage{wrapfig}
\usepackage{amsthm}

\theoremstyle{plain}% 
\newtheorem{theorem}{Theorem}
\newtheorem{proposition}[theorem]{Proposition}
\newtheorem{example}[theorem]{Example}
\newtheorem{lemma}[theorem]{Lemma}
\newtheorem{corollary}[theorem]{Corollary}
\newtheorem{definition}[theorem]{Definition}
\newtheorem{remark}[theorem]{Remark}

\newtheorem{algorithm}{Algorithm}

\date{}
\begin{document}
{
\title{Explicit Deformation of Lattice Ideals via\\ Chip-Firing Games
  on Directed Graphs}

\author{Spencer Backman\footnote{Spencer Backman was partially supported by NSF grant DMS-1201473, a Georgia Institute of Technology Algorithms and Randomness Center Student Fellowship, and by the European Research Council under the European Unions Seventh Framework Programme (FP7/2007-2013)/ERC Grant Agreement no. 279558 during this work.} and Madhusudan Manjunath\footnote{Madhusudan Manjunath was supported by the Feoder-Lynen Fellowship of the Humboldt Foundation during this work.} }

\maketitle

 %{\color{blue} The title and author's names should appear alternatingly at the top of each page}

For a finite index sublattice $L$  of the root lattice of type $A$, we construct a deterministic algorithm to deform the lattice ideal $I_L$ to a nearby generic lattice ideal, answering a question posed by Miller and Sturmfels.  Our algorithm is based on recent results of Perkinson, Perlman and Wilmes concerning commutative algebraic aspects of chip-firing on directed graphs.  As an application of our deformation algorithm,  we construct a cellular resolution of the lattice ideal $I_L$  by degenerating the Scarf complex of its deformation.

\section{Introduction}
Let $\Bbbk$ be an arbitrary field and let $\Bbbk[x_0,\dots,x_n]$ be the polynomial ring in $(n+1)$-variables.  Given a lattice $L \subseteq \mathbb{Z}^{n+1}$, the lattice ideal $I_L$ is a binomial ideal associated to $L$.  In particular, $I_L= \langle {\bf x^{\bf u}}-{\bf x^{\bf v}} | ~{\bf u-v} \in L , ~ {\bf u},~ {\bf v} \in \mathbb{N}^{n+1} \rangle$.  Lattice ideals generalize toric ideals and are among the most well-studied objects in combinatorial commutative algebra \cite{MilStu05}.  In particular, the problem of existence and construction of cellular minimal free resolutions for lattice ideals, i.e., resolutions for  lattice ideals that are supported on a cellular complex, has been a source of immense interest in the recent past. 
 The problem of constructing a cellular resolution for any lattice ideal was solved by the hull complex developed by Bayer and Sturmfels \cite{BayStu98}. The hull complex is in general not a minimal free resolution and the problem of existence and construction of {\bf minimal} cellular free resolutions  of lattice ideals remains open.  Sturmfels and Peeva \cite{PeevaStu98} define a notion of {\bf generic} lattice ideals and construct a cellular minimal free resolution called the Scarf complex for generic lattice ideals.  A lattice ideal is called generic if it is generated by binomials ${\bf x^{\bf u}}-{\bf x^{\bf v}}$ such that the vector {\bf u-v} has full support. The term ``generic''  is justified by a theorem of Barany and Scarf that the lattices corresponding to generic lattice ideals are dense in the space of all lattices \cite{BarSca98}. Miller and Sturmfels in their book on Combinatorial Commutative Algebra \cite[Page 189]{MilStu05}  remark that despite this abundance of generic lattice ideals, most lattice ideals one encounters in commutative algebra seem to be nongeneric and they ask for a deterministic algorithm to deform an arbitrary lattice ideal into a nearby ``generic'' lattice ideal.  Our main result in this paper, presented in Section \ref{deform_sect},  is an algorithm (cf. Algorithm \ref{deform_alg}) to deform a lattice ideal $I_L$ where $L$ is a finite index sublattice of $A_n$ and is summarized by the following theorem: 

\begin{theorem}
Let $A_n=\{(x_0,\dots,x_n) \in \mathbb{Z}^{n+1}|~ \sum_{j=0}^n x_j=0\}$ be the root lattice of type $A$ and rank $n$.
 Given a lattice ideal $I_L$, where $L$ is a finite index sublattice of $A_n$, and a real number $\delta>0$, the deterministic algorithm,  Algorithm \ref{deform_alg} (in Subsection \ref{deformalg_subsect}) outputs a generic lattice ideal $I_{\lambda \cdot L_{\rm{gen}}}$ such that the lattice $L_{\rm gen}$ is distance at most $\delta$ from $L$, where the distance is measured by a metric in the space of sublattices of $\mathbb{R}^{n+1}$ (see Definition \ref{latdist_def}) and the scaling $\lambda \cdot L_{\rm gen}$ of $L_{\rm gen}$ is a sublattice of $\mathbb{Z}^{n+1}$. 
 \end{theorem}

Our deformation algorithm is based on the recent results of Perkinson-Perlman-Wilmes. \cite{Perkinson11} on lattice ideals corresponding to the lattice generated by the rows of the Laplacian matrix of a directed graph. In particular, they showed that every full rank sublattice of $\mathbb{Z}^n$ is generated by the rows of the reduced Laplacian of a directed graph.  This allows us to reduce the problem of deforming lattice ideals  to the problem of deforming Laplacian lattice ideals of directed graphs.  Perkinson-Perlman-Wilmes.  in the same paper also described a certain distinguished Gr\"obner basis of the Laplacian lattice ideal of a directed graph whose equivalent chip-firing interpretation was independently discovered by the first author and Arash Asadi \cite{AsaBac}. The key observation about these Gr\"obner bases which  we exploit in our deformation algorithm is that, in a precise sense, they respect certain perturbations of the lattice.  
In Section \ref{res_sect},  we use our deformation algorithm to construct a cellular free resolution. In fact, the free resolution we construct is supported on a simplicial complex. Recall that the hull complex is also a (non-minimal) cellular free resolution of any lattice ideal $I_L$.  For any lattice ideal $I_L$ where $L$ is a finite index sublattice of $A_n$, the cellular resolution we construct is an alternative to the hull complex.  More precisely, we show the following:

\begin{theorem} (cf. Theorem \ref{Scarfdef_theo}) For any finite index sublattice $L$ of $A_n$, the complex of free $\Bbbk[x_0,\dots,x_n]$-modules obtained from the labelled simplicial complex ${\rm Scarf_{def}}(I_L)$ is supported on a simplicial complex and is a free, in general non-minimal, resolution of  the lattice ideal $I_L$.\end{theorem}
 
 The complex ${\rm Scarf_{def}}(I_L)$ is constructed by degenerating the Scarf complex of a deformation of $I_L$.  We remark that the minimal free resolution of undirected Laplacian lattice ideals has received significant attention in recent years \cite{Doc, Hop, ManSchWil, MilStu05, MohSho, Pos}, and that deformations of undirected Laplacian lattice ideals were very recently applied to the study of questions arising in algebraic statistics \cite{Kat}.
 
\section{Chip-firing on Directed Graphs}

Let $\vec{G}$ be a directed graph with vertex set $\{v_0, ..., v_n \}$ and adjacency matrix $A$ whose entry $A_{i,j}$ for $0\leq i,j \leq n$ is the number of edges directed from $v_i$ to $v_j$.  Let $V(\vec{G})$ and $E(\vec{G})$ be the vertex set and edge set of $\vec{G}$ respectively. Let ${\mathcal{D}}=diag(\vec{\deg}(v_0), \dots, \vec{\deg}(v_n))$ where $\vec{\deg}(v)$ denotes the number of edges leaving vertex $v \in V(\vec{G})$. We call the matrix $Q=\mathcal{D}-A$ the {\bf Laplacian matrix} of the directed graph $\vec{G}$.  We note that this definition of the Laplacian is the transpose of the Laplacian appearing in the work of  Perkinson-Perlman-Wilmes. \cite{Perkinson11}.

 We now describe the associated chip-firing game on the vertices of $\vec{G}$ coming from the rows of the Laplacian matrix.  Let ${\bf C} \in \mathbb{Z}^{n+1}$, which we call a {\bf chip configuration} whose $i$th coordinate $C_i$ is the number of chips at vertex $v_i$.  We say that a vertex {\bf fires} if it sends a chip along each of its outgoing edges to its neighbors.  We say that a vertex $v_i$ is in {\bf debt} if $C_i<0$.  Note that the process is ``commutative'' in the sense that the order of firings does not affect the final configuration.  For ${\bf f} \in \mathbb{Z}^{n+1}$, we interpret the configuration ${\bf C'}={\bf C}-Q^T{\bf f}$ as the configuration obtained from ${\bf C}$ by a sequence of moves in which the vertex $v_i$ fires $f_i$ times, and we call ${\bf f}$ a {\bf firing}.  We restrict our attention to {\bf strongly connected directed graphs}, directed graphs for which there is a directed path between every ordered pair of distinct vertices. The following lemma is an algebraic characterization of the strongly connected property.  
 
 %%This characterization has a natural chip-firing interpretation.
\begin{lemma}
\label{Rightkernel_lem}
A directed graph $\vec{G}$ is strongly connected if and only if there exists a row vector ${\bf \Sigma}^T=(\Sigma_0,\dots,\Sigma_n)$ with strictly positive integer entries that spans the left kernel of $Q$.
\end{lemma}
\begin{proof}

Let $\vec{G}$ be strongly connected.   By construction, $Q \bf{1} = 0$ where ${\bf 1}=(1,\dots,1)$, which just says that directed chip-firing moves preserve the total number of chips in the graph, therefore the Laplacian is not of full rank and has some nontrivial left kernel.  Given two vectors ${\bf v_1}$ and ${\bf v_2}$ we say that ${\bf v_1} > {\bf v_2}$ when each coordinate of ${\bf v_1}$ is strictly greater than the corresponding coordinate of ${\bf v_2}$.  Suppose that there exists some firing strategy ${\bf \Sigma}$ with ${\bf \Sigma} \not > {\bf 0}$ and ${\bf \Sigma} \not < {\bf 0}$ such that ${\bf \Sigma}$ has no effect on chip configurations, i.e., such that  ${Q}^T{\bf \Sigma}={\bf 0}$.  Let $V^+$ be the set of vertices of $\vec{G}$ such that $\Sigma_i > 0$ for every vector $v_i \in V^+$.  We may assume that $V^+ \neq \emptyset$ by taking the negative of ${\bf \Sigma}$ if necessary.  Since the net amount of chips leaving $V^+$ is positive,  there must exist some integer $j$  such that $v_j \in V^{+}$ and $({Q}^T{\bf \Sigma})_j<0$, a contradiction. Assume that there exist two linearly independent firing strategies ${\bf f_1}>0$ and ${ \bf f_2}>0$, such that ${Q}^T{\bf f_1}={\bf 0}$ and ${Q}^T{\bf f_2}={\bf 0}$, then there exists a non-zero linear combination $\lambda_1{\bf f_1}+ \lambda_2{\bf f_2} \not> {\bf 0}$, which is in the kernel, a contradiction.

Conversely, suppose that $\vec{G}$ is not strongly connected, but that there exists some firing strategy ${\bf \Sigma} > {\bf 0}$ such that $Q^T\bf {\Sigma}={\bf 0}$. Let $V_1, \dots, V_t$ be the partition of vertices of $\vec{G}$ into maximal strongly connected components.  We construct a graph with vertices $V_1,\dots,V_t$  and an edge between 
$(V_i,V_j)$ if there exists $v_i \in V_i$ and $v_j \in V_j$ with $(v_i,v_j) \in E(\vec{G})$. This meta graph has at least two vertices since $\vec{G}$ is not strongly connected.  Furthermore, it is acyclic since otherwise we could find a larger strongly connected component.  Hence,  there exists some source vertex $V_i$, i.e., some component with no edges $(u,v)$ in $\vec{G}$ where $u \in V_i$, $2 \leq j \leq t$ and $v \in V_i$.  The total number of chips leaving $V_i$ is positive, therefore there must exist some vertex $v_k \in V_i$ such that $(Q^T{\bf \Sigma})_k<0$, a contradiction.
\end{proof}

%%\subsection{A Chip-Firing Gr\"obner Basis for Full Dimensional Lattice Ideals}

Recall that the reduced Laplacian matrix of a directed graph is the Laplacian matrix with the zeroth row and the zeroth column deleted.

\begin{theorem}{\rm(Perkinson-Perlman-Wilmes \cite[Theorem 5.13]{Perkinson11})}\label{Latred_theo}
Fix any integer $m>0$,  every full rank sublattice of $\mathbb{Z}^m$ has a basis whose elements are the columns of a reduced Laplacian matrix of a directed graph.   Furthermore, we can take this graph to be such that $v_0$ is globally reachable meaning that there is a directed path from $v_i$ to $v_0$ for each $i$.
\end{theorem}

When $L$ is a full-rank sublattice of the root lattice $A_n$,  Theorem \ref{Latred_theo} can be reformulated as follows: there exists a strongly connected directed graph with $\Sigma_0 = 1$ whose rows of the Laplacian form a basis for $L$. Take a basis $B$ for $L$, and apply the basis algorithm of Perkinson-Perlman-Wilmes. \cite{Perkinson11} to these vectors by ignoring the first coordinate.  Each vector in the resulting generating set $B'$ has a nonpositive first entry because the sum of the coordinates is zero.  Letting $\bf M$ represent the matrix whose rows are the vectors in $B'$ and taking $\sigma$ to be the minimum script vector, the vector $v = -\sigma \bf M$ is positive in the first entry and non positive in the remaining entries.  Including $v$ as the first row of $\bf M$, we obtain a set of vectors coming from the rows of a strongly connected digraph (as described in Theorem \ref{Latred_theo}) because the left kernel is positive and one-dimensional.

The following theorem is important for a combinatorial description for a Gr\"obner basis for $I_L$.
\begin{theorem}{\rm(Asadi -Backman \cite[Corollary 3.7 and Corollary 3.9]{AsaBac}, Perkinson-Perlman-Wilmes. \cite[Theorem 5.11]{Perkinson11})}\label{Grochip_theo}
Let $\vec{G}$ be a strongly connected directed graph  with Laplacian $Q$, let ${\bf \Sigma}=(\Sigma_0,\dots,\Sigma_n)$ be a positive vector in the left kernel of $Q$ such that the entries of ${\bf \Sigma}$ are relatively prime. Let ${\bf C}$ be a configuration of chips which is nonnegative away from $v_0$.  Any sequence of firings ${\bf f}_0,{\bf f}_1,\dots$  that satisfies the following properties:
\begin{itemize}
\item  Each firing is non-zero and satisfies $\bf{0} \leq \bf{f_j} \leq \bf{\Sigma}$, 
\item  For all $j$, we have $({\bf f_j})_0 = 0$, 
\item  No vertex is sent into debt.
\end{itemize}
is a  finite sequence. Furthermore, the final configuration is independent of the order in which the firing is made.  
\end{theorem}

These configurations of chips obtained by the process from the previous theorem are referred to in the literature as $v_0$-reduced divisors or superstable configurations.  Let $T$ be a directed spanning tree all of whose edges are directed towards a root $v_0$.  Given $u, v \in V(G)$, we say that $u \leq v$ in the spanning tree partial order associated to $T$ if there is a directed path from $v$ to $u$ in $T$.

%, or  $\vec{G}$-parking functions.

\begin{corollary}{(Perkinson-Perlman-Wilmes. \cite[Theorem 5.11]{Perkinson11})\label{Gro_cor}}
The set $\tilde G = \{  \bf{x^{u^+}} -  \bf{x^{u^-}} , u =Q^T {\bf x} , \bf{0} \leq \bf{x} \leq \bf{\Sigma} \}$ is a grevlex Gr\"obner basis for $I_L$ for any linear extension of a rooted spanning tree order with root $v_0$. 
\end{corollary}

In the undirected case Cori-Rosen-Salvy \cite{CorRosSal02} showed that Theorem \ref{Grochip_theo} translates into the statement that the binomials defined by the cuts in the graph, which correspond to firing moves, are a Gr\"obner basis for the Laplacian lattice ideal  with respect to any linear extension of a spanning tree term order.  This is due to the fact that one characterization of a Gr\"obner basis is a generating set such that division with respect to the given term order is unique.  We remark that the exponent vectors of the standard monomials correspond to the $v_0$-reduced divisors, also known as superstable configurations.  Perkinson-Perlman-Wilmes. \cite{Perkinson11} observed that this result of Cori-Rosen-Salvy \cite{CorRosSal02} extends to the case of directed graphs where $\Sigma_0=1$ via Theorem \ref{Grochip_theo}.  This Gr\"obner basis is not minimal in general.

A slight difference in the approaches of \cite{AsaBac} and \cite{Perkinson11} is that the latter more often work with the reduced Laplacian of a graph where $v_0$ is globally reachable (a sandpile graph) while the former work with the full Laplacian of a strongly connected directed graph.  In \cite{AsaBac}, the authors worked with general strongly connected graphs for investigation of Riemann-Roch theory for directed graphs. Translating between the settings of  \cite{AsaBac} and \cite{Perkinson11} requires a little finesse. In contrast to the case of undirected graphs where passage between reduced and full Laplacians is completely transparent as both the row and columns sums are zero,  in the directed case the column sums are not necessarilly zero:  the case when both the row and columns sums are zero corresponds to the situation when our directed graph is Eulerian.  If we take the lattice generated by the rows of the reduced Laplacian $Q$ of a digraph with $v_0$ globally reachable and then homogenize with respect to the coordinate corresponding to $v_0$, the sink vertex, we obtain a full-rank sublattice of $A_n$.  This lattice is generated by the rows of the full Laplacian of a strongly connected directed graph where the row corresponding to $v_0$ can be canonically obtained as $-Q^T\sigma$, where $\sigma$ is the {\it minimal script vector}.  We do not explain the minimal script vector for the reduced Laplacian but refer to \cite{Perkinson11} for its precise definition, and note that if one has the full Laplacian, it is the vector obtained from $\bf{\Sigma}$ by deleting first entry.  On the other hand, if we take the lattice generated by the rows of the full Laplacian of a strongly connected directed graph,  dehomogenization with respect to the coordinate  corresponding to the vertex $v_0$ gives the lattice spanned by the reduced Laplacian if and only if ${\bf \Sigma}_0=1$.  Using a simple variant of Theorem \ref{Latred_theo},  we can always take a generating set for the same lattice coming from a different strongly connected directed graph with ${\bf \Sigma}_0 = 1$.

\begin{example}\label{dirgraph_ex}
Consider the directed graph $\vec{G}$ shown in Figure \ref{digraph}. It has Laplacian matrix:
\begin{figure}\hspace{3.5cm}
 \includegraphics{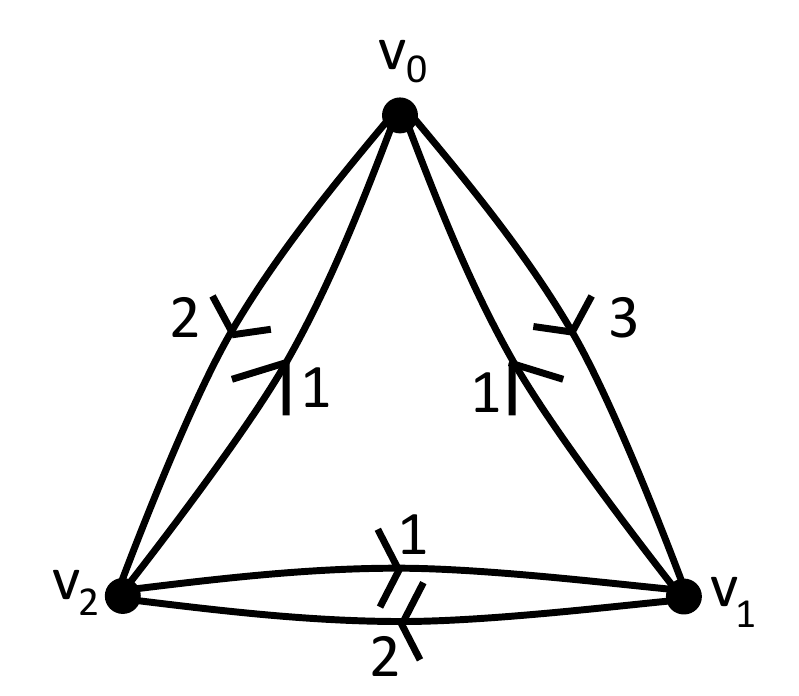} 
 \caption{Example}
 \label{digraph}
\end{figure}

$$Q =
\begin{pmatrix}
5  & -3  & -2\\
-1  & 3 &  -2 \\
-1   & -1  & 2
 \end{pmatrix}$$.
 
 The left kernel $\bf{\Sigma}$ of $Q$ is $(1,2,3)$.  The Laplacian lattice ideal $I_L$, i.e., the lattice generated by the rows of $Q$ reads

  $I_L=\langle x_2^{2}-x_1x_0, x_2^4-x_0^2x_1^2,x_2^6-x_0^3x_1^3, x_1^3-x_0x_2^2, x_1^2-x_0^2, x_1x_2^2-x_0^3,  x_2^4-x_0^4, x_1^6-x_0^2x_2^4, x_1^5-x_0^3x_2^2, 
 x_1^4-x_0^4,x_1^3x_2^2-x_0^5
 \rangle$.  

In fact, the above generating set is a Gr\"obner basis of $I_L$ with respect to  any order with $x_0$ minimum.  Since the exponent vectors $(-2,2,0)$,  $(4,0,-4)$ and $(-4,4,0)$ of the binomials  $x_1^2-x_0^2$, $x_2^4-x_0^4$ and $x_1^4-x_0^4$ do not have full support, the lattice ideal is not evidently generic. In Example \ref{deform_ex}, we deform $I_L$ into a generic lattice ideal.\qed
 \end{example}

\section{Explicit Deformation of Lattice Ideals}\label{deform_sect}
\subsection{Deformation of Lattice Ideals} \label{Laqseq_subsect}

 We start by defining a metric on the space of lattices and use this metric to give a precise definition of a deformation of a lattice ideal.

\begin{definition}\rm{({\bf Metric on the Space of Lattices and Convergence on Lattices)}}\label{latdist_def}
For  sublattices $L_1$ and $L_2$ of $\mathbb{R}^n$,  we define the {\bf distance} $d(L_1,L_2)$ between $L_1$ and $L_2$ as the minimum of $||B_1-B_2||_2$ over all bases $B_1$ and $B_2$ of $L_1$ and $L_2$ respectively expressed as matrices and $||.||_2$ is the $\ell_2$-norm on matrices.   A sequence of lattices $\{L_k\}$ is said to {\bf converge} to a lattice $L_{\ell}$ if for every $\delta>0$ there exists a positive integer $K(\delta)$ such that $d(L_k,L_{\ell}) \leq \delta$ for all $k \geq K(\delta)$.
\end{definition}

See the book of Cassels \cite[Page 127]{Cass59} for a detailed discussion on sequences of lattices and a proof that $d(.  , .)$ is a metric on the space of sublattices of $\mathbb{R}^n$.   Barany and Scarf in \cite{BarSca98}  show that the generic lattices are dense in the spaces of all lattices with respect to the topology induced by this metric.

\begin{definition}\rm{{\bf (Deformation of a Lattice Ideal)}} Given a lattice $L \subset \mathbb{Z}^{n+1}$, a deformation of $I_L$ is a sequence of lattices $\{L_{k}\}$ that converges to $L$ and such that for every lattice $L_i$ in the sequence, there is a non-zero real number $\lambda_i$ such that  $\lambda_i \cdot L_i$ is a sublattice of $\mathbb{Z}^{n+1}$ and the lattice ideal $I_{\lambda_i \cdot L_i}$ is a generic lattice ideal.  We call the sequence $\{ I_{\lambda_k \cdot L_k}\}$ a deformation of the lattice ideal $I_L$. Fix  $\delta>0$,  a lattice $L_{\rm gen}$ is called a $\delta$-deformation of $L$ if  $d(L,L_{\rm gen})=\delta$ and there exists a $\lambda \in \mathbb{R}$ such that $\lambda \cdot L_{\rm gen} \subseteq \mathbb{Z}^{n+1}$ is a generic lattice. We call $\delta$ the deformation parameter.

\end{definition}

\subsection{A Naive Approach to Explicit Deformation}\label{firstapp_subsect}

A natural first approach, but an unsucessful one, to deform a lattice ideal $I_L$ would be to take an arbitrary generating set of $I_L$ and deform its exponents to obtain an ideal generated by binomials all of whose exponents have full support.  The pitfall to this approach is that the resulting ideal need not be a lattice ideal as the following example shows: consider the lattice ideal $I_{A_4}$ where $A_4$ is the root lattice of type $A$ and rank three. The ideal $I_{A_4}$ is minimally generated by the binomials 
$x_0-x_1$, $x_1-x_2$ and $x_2-x_3$.  One deformation of these binomials is $x_0x_2^{\epsilon}x_3^{\epsilon}-x_1^{1+2\epsilon}$, $x_1x_0^{\epsilon}x_3^{\epsilon}-x_2^{1+2\epsilon}$ and $x_2x_1^{\epsilon}x_0^{\epsilon}-x_3^{1+2 \epsilon}$ respectively where $\epsilon=1/k$ for some large integer $k$.  Scaling the exponent vector of each binomial by $k$ we obtain $x_0^{k}x_2x_3-x_1^{k+2}$, $x_1^kx_0x_3-x_2^{k+2}$ and $x_2^kx_1x_0-x_3^{k+2}$. But, for any natural number $k$  the ideal $I_{k}$ generated by these binomials is not saturated with respect to the product of all the variables and is hence not a lattice ideal \cite[Lemma 7.6]{MilStu05}. To see that the ideal is not saturated, note that $x_1^2x_2^{k+1}-x_0^{k+1}x_3^2$ is a binomial in the saturation of $I_{k}$ with respect to the product of all variables. This is because the vector $(-k-1,2,k+1,-2)$ is a point in the lattice generated by the exponents of the binomial generators of $I_k$. But $x_1^2x_2^{k+1}-x_0^{k+1}x_3^2$ is not contained in $I_k$ since $x_1^2x_2^{k+1}$ is not divisible by any monomial term in the binomial generators $x_0^{k}x_2x_3-x_1^{k+2}$, $x_1^kx_0x_3-x_2^{k+2}$ and $x_2^kx_1x_0-x_3^{k+2}$ of $I_k$.  In fact, a key step in our deformation algorithm  is to deform the exponents of a (Gr\"obner) basis of the lattice ideal in such a way that the resulting ideal remains a lattice ideal.

\subsection{Deformation Algorithm}\label{deformalg_subsect}

\begin{algorithm}\label{deform_alg}
Deformation Algorithm
\end{algorithm}

{\bf Input:} A lattice ideal $I_L$ where $L$ is a full rank sublattice of $A_n$ and a real number $\delta>0$.

{\bf Output:} A Gr\"obner basis of a generic Laplacian lattice ideal $I_{\lambda \cdot L_{\rm{gen}}}$ where $L_{\rm{gen}}$ is a generic lattice and such that $d(L_{\rm gen},L) \leq \delta$ and ${\lambda \cdot L_{\rm{gen}} \subseteq \mathbb{Z}^{n+1}}$.

\

As a preprocessing step, apply the lattice reduction algorithm from Theorem \ref{Latred_theo} to compute a strongly connected directed graph $\vec{G}$ whose corresponding Laplacian $Q$ has rows generating $L$ and has left kernel ${\bf \Sigma}$ with $\Sigma_0=1$. \\

Begin with $Q_0:=Q$ and at the $r$th iteration of this step, let $Q_r$ be our current Laplacian matrix.  Let $\lambda_r \in {\mathbb{Q}}_{\geq0}$ be the minimum value such that $\lambda_r Q_r$ has integral entries.  If the Gr\"obner basis given by Corollary \ref{Gro_cor} applied to $\lambda_r Q_r$ has full support, then set $Q_{r} = Q_{gen}$ and $\lambda = \lambda_r$, output this Gr\"obner basis, and take $L_{gen}$ to be the lattice spanned by the rows of $Q_{gen}$. 

 If this Gr\"obner basis constructed by Corollary \ref{Gro_cor} does not have full support, then there exists some vector $\bf{x}$ with ${\bf 0 \lneq x \lneq \Sigma}$ and some index $i$ such that $((\lambda_r Q_r)^T  {\bf x})_i=0$.  Find a coordinate $j$ such that ${x_j / \Sigma_j}  \neq  {x_i / \Sigma_i} $, which exists since ${\bf \Sigma}$ is a primitive vector and ${\bf x} \lneq {\bf \Sigma}$.  Let $\hat{Q}_r$ be the Laplacian matrix of the directed graph $\vec{H}_{i,j}$ on $(n+1)$-vertices labelled $v_0,\dots,v_n$ with two directed edges, one from $v_i$ to $v_j$  of weight $1/ \Sigma_i$ and one from $v_j$ to $v_i$ of weight $1 / \Sigma_j$.   Take $Q_{r+1}=Q_r+\epsilon_r \hat{Q}_r$ with $\epsilon_r \in {\mathbb{Q}}_{\geq0}$ and %

\begin{equation}\label{bounds_eq}
\epsilon_r < - \min_{ \{ 0\leq y \leq {\bf \Sigma} ,  (\hat Q_r ^T y)_k  ( Q_r^T y)_k < 0, k \in \{i, j\}\}}   \{ {( Q_r^T y)_k \over (\hat Q_r ^T y)_k}, {\delta \over |\hat Q_r|_1 (n+1) \prod _s \Sigma _s} \}.
\end{equation}

\begin{theorem}\label{verifyalg}
 Given any lattice ideal $I_L$, where $L$ is a full rank sublattice of $A_n$ and any real number $\delta>0$, Algorithm \ref{deform_alg}  outputs a generic lattice ideal $I_{ L_{\rm{gen}}}$ such that $d(L_{\rm gen},L) \leq \delta$. \end{theorem}

\begin{proof}

We first show that the Algorithm \ref{deform_alg} terminates.  In every iteration, by choosing $\epsilon_r$ smaller than the first term in the minimum, for every vector ${\bf 0} \leq  {\bf z} \leq {\bf \Sigma}$ such that  $(Q_r{\bf z})_j \neq 0$,  we have $(Q_{r+1}{\bf z})_j \neq 0$.  Moreover, there exists a vector ${\bf 0} \leq {\bf y}\leq {\bf \Sigma}$ and an index $i$ such that $(Q_{r} {\bf y})_i=0$ and $(Q_{r+1}{\bf y})_i \neq 0$.   Hence, the algorithm terminates after at most $(\prod_j \Sigma_j)(n+1)$ iterations.  %{\color{blue} More Details?}

To show that the left kernel of  $Q_{r}$ is spanned by $\Sigma$, we proceed by induction on $r$.  The statement is true for $Q_0$, so we assume it is true for $Q_r$ and verify the statement for $Q_{r+1}$.  The vector $\Sigma$ is also in the left kernel of $\hat Q_r$, thus it is contained in the left kernel of $Q_{r+1}=Q_r+\epsilon_r \hat{Q}_r$.  The matrix $\lambda_{r+1} Q_{r+1}$ is the Laplacian of a directed graph obtained from $\vec{G}$ by adding edges, hence this graph is strongly connected, and the kernel is one-dimensional by Lemma \ref{Rightkernel_lem}, so the same holds for $Q_{r+1}$.  Thus $\Sigma$ spans the left kernel of $Q_{r+1}$.  By Theorem \ref{Latred_theo} and Corollary \ref{Gro_cor}, we obtain a Gr\"obner basis for a Laplacian lattice ideal $I_{L_{{\rm gen}}}$.  Moreover, this Gr\"obner basis has full support because $(\lambda {Q_{\rm gen}}^T  {\bf x})_i\neq0$ (recall that $Q_{\rm gen}$ is the Laplacian matrix associated to $L_{\rm gen}$)  for all ${\bf 0 \lneq x \lneq \Sigma}$ and indices $0 \leq i \leq n$ from which it follows that $I_ {L_{{\rm gen}}}$ is generic.  Since we chose $\epsilon_r$ at each step to be less than the second term in the minimum, every entry in $Q_{{\rm gen}}-Q$ has absolute value at most $\delta/{n+1}$. Hence (by Definition \ref{latdist_def}) $d(L, L_{{\rm gen}}) \leq \delta$.   We note that the second term in the minimum makes use of our knowledge that the algorithm terminates in at most $(n+1)\prod_s \Sigma_s$ iterations.
\end{proof}

Algorithm \ref{deform_alg} generalizes to sublattices of $A_n$ of lower rank as follows:  given any sublattice $L$ of $A_n$, we first extend a basis $B$ of the lattice $L$ to a basis $\tilde{B}$ of a lattice $\tilde{L} \subset \mathbb{Q}^n$ of rank $n-1$ such that $||B-\tilde{B}||_2 \leq \delta/2$. Apply Algorithm \ref{deform_alg} to the lattice $\ \tilde{L}$ with parameter $\delta/2$. 
Note that this approach is somewhat unsatisfactory for lattices of rank strictly smaller than $n$ since the rank of the deformed lattice  will not be the same as the given lattice. In particular the cellular resolution from Section \ref{res_sect} does not apply.

If we impose the restriction that the deformed lattice has the same rank as the given lattice then our approach does not directly generalize as the following example shows. Consider the sublattice spanned by the vector $(1,1,-1,-1)$ of $A_3$. Since every element in this lattice has two positive coordinates, this lattice is not generated by the rows of the Laplacian of a directed graph. Hence,  Theorem \ref{Latred_theo} and Algorithm \ref{deform_alg} do not  apply to this lattice. 

\begin{figure}\hspace{3.5cm}
 \includegraphics{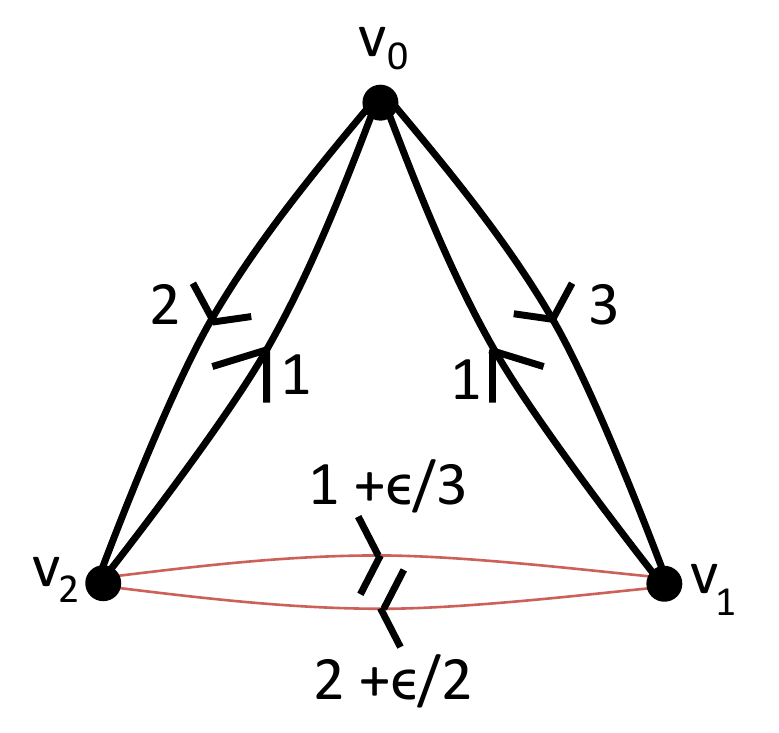} 
 \caption{Example}
 \label{deformation}
\end{figure}

\begin{example}\label{deform_ex}
In Example \ref{dirgraph_ex}, the generating set of the lattice ideal $I_L$ has three binomials $x_1^2-x_0^2$, $x_1^4-x_0^4$ and $x_2^4-x_0^4$ whose exponents do not have full support. Following the explicit deformation algorithm (Algorithm \ref{deform_alg}), we deform the Laplacian ${\vec Q}
$ to the digraph shown in Figure \ref{deformation} to ${\vec Q}$:\\

$$
\vec{Q}^{\epsilon}=
\begin{pmatrix}
5  & -3  & -2\\
-1  & 3+\epsilon/2 &  -2-\epsilon/2 \\
-1   & -1-\epsilon/3  & 2+\epsilon/3
 \end{pmatrix}$$.

\

Suppose $\epsilon=p/q$, for non-zero natural numbers $p$ and $q$.  We can scale this matrix by the integer $6q$ to obtain a matrix with integer entries\\

$$
6q \cdot \vec{Q}^{\epsilon}=
\begin{pmatrix}
30q  & -18q  & -12q\\
-6q  & 18q+3p &  -12q-3p \\
-6q   & -6q-2p  & 12q+2p
 \end{pmatrix}$$.

\

 This matrix is the Laplacian matrix of a directed graph.  Let  $L_{\rm gen}$ be the lattice generated by the rows of the matrix $6q\cdot \vec{Q}^{\epsilon}$. The lattice $I_{\lambda L_{\rm gen}}$ ideal is:\\

$I_{L_{\rm gen}}=\langle x_2^{12q+2p}-x_0^{6q}x_1^{6q+2p}, x_2^{24q+4p}-x_0^{12q}x_1^{12q+4p}, x_2^{36q+6p}-x_0^{18q}x_1^{18q+6p}, x_1^{18q+3p}-x_0^{6q}x_2^{12q+3p}, x_1^{12q+p}-x_0^{12q}x_2^{p}, x_2^{12q+p}x_1^{6q-p}-x_0^{18q}, x_2^{24q+p}-x_0^{24q}x_1^{p}, x_1^{36q+6p}-x_0^{12q}x_2^{24q+6p}, x_1^{30q+4p}-x_0^{18q}x_2^{12q+4p},  x_1^{24q+2p}-x_0^{24q}x_2^{2p}, x_1^{18q}x_2^{12q}-x_0^{30q} \rangle$.

The above generating set is a Gr\"obner basis of $I_{L_{\rm gen}}$ as described in Corollary \ref{Gro_cor}.  The lattice ideal $I_{L_{\rm gen}}$ is generic and $I_{L_{\rm gen}}$ is a $\delta$-deformation of $I_L$.  The choice of the deformation (Figure \ref{deformation}) is not unique. In this example, adding the Laplacian matrix of one directed graph to $\vec{Q}$ was sufficient for the deformation, but in general we might have to perform this operation several times.  \end{example}

\subsection{Geometric Aspects of the Explicit Deformation}\label{Geom_subsec}
The set of Laplacians associated to real edge weighted strongly connected directed graphs with a one-dimensional left kernel spanned by a vector ${\bf \Sigma} \in {\mathbb{N}^{n+1}}$ is  a cone of matrices viewed as points in $\mathbb{R}^{{(n+1)}^2}$.  We denote this cone by $C_{\Sigma}$. The cone $C_{\Sigma}$ is polyhedral, lying in the space of matrices with row sums equal to zero and has facet defining inequalities given by the nonpositivity constraints of the  off-diagonal entries.  The cone $C_{\Sigma}$ is neither closed nor open and its interior is given by the collection of Laplacians of saturated graphs, i.e., those with no non-zero entry and whose left kernel is spanned by ${\bf \Sigma}$.  For each Laplacian matrix in $C_{\Sigma}$, we associate a lattice ideal as follows: this is the (lattice) ideal corresponding to the lattice generated by the rows of the Laplacian matrix.  Theorem \ref{verifyalg} implies that the subset of Laplacian matrices whose lattice ideal is generic is dense in $C_{\Sigma}$.  The closure of $C_{\Sigma}$ is the set of all Laplacians with left kernel containing ${\bf \Sigma}$, i.e., those Laplacians coming from digraphs comprised of a vertex disjoint collection of strongly connected components, whose left kernel is given by the corresponding restriction of ${\bf \Sigma}$. 

\begin{proposition}
The rays of $C_{\Sigma}$ are generated by the Laplacians coming from weighted cycles obtained in the following way:  take some subset $S$ of $V(\vec{G})$ and cyclically order the elements of $S$.  If $v_i$ is followed by $v_j$ in the cyclic order, add an edge $(v_i,v_j)$ with weight $1 \over \Sigma_i$. \end{proposition}
\begin{proof}
 First observe that we can scale the $i$th row in a Laplacian by $\Sigma_i$ to linearly map $C_{\Sigma}$ to the collection of Laplacians with left kernel spanned by $\bf {1}$.  These are the Laplacians of Eulerian directed graphs, i.e., those having an Eulerian circuit, and it is a classical fact that such directed graphs decompose into directed cycles, allowing cycles of length 2.  We map these directed cycles back to the aforementioned weighted directed cycles by rescaling the rows.
\end{proof}

\section{Free Resolutions of  Lattice Ideals By Degeneration}{\label{res_sect}}

%%%%%Mention what you do and mention the hull resolutions. 

 We construct a (non-minimal) cellular free resolution of the lattice ideal $I_L$ arising from a finite index sublattice of $A_n$.  This free resolution is constructed by degenerating the Scarf complex of a deformation of the lattice ideal. Algorithm \ref{deform_alg} can be used to perform this deformation.  Our construction is an adaptation of the  free resolutions of monomial ideals constructed from a generic deformation in \cite[Section 6.3]{MilStu05} to lattice ideals. We start with the framework of Bayer and Sturmfels \cite{BayStu98}.  Bayer and Sturmfels \cite{BayStu98} consider a Laurent monomial module associated with a lattice ideal. Note that Laurent monomial modules associated with a lattice are categorically equivalent to lattice ideals. We then deform the exponents of this Laurent monomial module such that the resulting lattice ideal is generic. The Scarf complex of a deformed lattice ideal is hence a minimal free resolution and  we then ``relabel'' this Scarf complex to obtain a free resolution for $I_L$.  Note that the ``relabeling'' procedure is described in detail below. Unlike in the case of monomial ideals in \cite[Section 6.3]{MilStu05}, the Laurent monomial module has infinitely many minimal generators and this makes the choice of the deformation parameter $\delta_0$ more involved.  We choose a sequence of deformations of the lattice ideal $I_{L}$, converging to $I_L$, all of which have the same Scarf complex (as an unlabeled simplicial complex).
 
 %%%%%Computing the Deformation Parameter $\delta$
We construct the deformation parameter $\delta_0$ using Theorem 5.4 of  \cite{PeevaStu98} that provides a description of the Scarf complex of $I_L$ in terms of the Scarf complex of the initial ideal with respect to a degree reverse lexicographic order. More precisely, we use the following restatement of Theorem 5.4 of  \cite{PeevaStu98}, we refer to the paper for the original statement: 

\begin{theorem} (Theorem 5.4 of \cite{PeevaStu98})\label{PeevaStu-rephrase}
The Scarf complex of a generic lattice ideal depends only on the lcm poset of any grevlex initial ideal. Hence there exist perturbations under which the Scarf complex of a generic lattice ideal is stable.
\end{theorem} 

In the following, we give a precise description of the deformation.  We construct a sequence of generic lattice ideals which converges to $I_L$ and with the same poset defined by the least common multiples of all non-empty subsets of the initial ideal.  We will refer to this poset as the lcm poset. The construction is as follows: take some $\delta$-deformation of the Laplacian matrix $Q$ and let $\hat{Q}_r$ for $0 \leq r \leq t$ be the set of Laplacians which are added in the process of deforming $Q$.  We describe a sequence of $\delta$-deformations of $Q$ which converge to $Q$ all having the same associated unlabeled Scarf complex.  First note that when we add $\epsilon_0 \hat{Q}_0$ to $Q_0$, if we take $\epsilon_0$ small enough taken according to (\ref{bounds_eq}), the Scarf complex of the grevlex initial ideal stabilizes.  Then if we add a small enough multiple $\epsilon_1$ of $Q_1$, the Scarf complex of the grevlex initial ideal again stabilizes.  Proceeding in this way, we obtain a new deformation $\tilde Q$, which we will use as a template for our sequence $\tilde Q_i$ of deformations.  Let $\tilde Q_i$ be the deformation obtained in the previous way but starting with $\epsilon_0 \leq 2^{-i}$ (or any other converging sequence).  Each deformation $\tilde Q_i$ will have a grevlex initial ideal with the same lcm poset.

%%%%%Relabelling Procedure.
We now describe the relabeling procedure to construct a free resolution of $I_L$ from the Scarf complex of its deformation. Given a lattice ideal $I_L$, we first deform the lattice ideal $I_L$ into a generic lattice ideal $I_{L_{\rm gen}}$ using the Algorithm \ref{deform_alg}.  Let $B=\{ {\bf b_0},\dots, {\bf b_{n-1}}\}$ and $B_{\rm gen}=\{{\bf b_{0, {\delta}}},\dots, {\bf b_{n-1, \delta}}\}$ be the first $n$ rows of the Laplacian matrix of a directed graph whose rows generate the lattice $L$ and its deformation $L_{\rm{gen}}$ respectively.  We construct the Scarf complex  of the generic lattice $L_{\rm gen}$ as described in \cite{BayStu98}. By construction, the vertices of the Scarf complex are precisely the points of $L_{\rm gen}$.  We relabel the vertices of the Scarf complex of $L_{\rm gen}$ with Laurent monomials ${\bf x^{\alpha}}$ where $\alpha$ is a point in $L$. In particular, we relabel a vertex  of the Scarf complex of $L_{\rm gen}$ by a point in $L$ that is ``close'' to the point in $L_{\rm gen}$ corresponding to this vertex. More precisely, suppose that ${\bf \alpha_{\delta}}$ is the lattice point given by ${\bf \alpha_{\delta}}=\sum_{k=0}^{n-1} \alpha_k {\bf b_{k, {\delta}}}$. We label  the vertex   ${\bf \alpha_{\delta}}$ of the Scarf complex by ${\bf x^{\alpha}}$, where ${\bf \alpha}=\sum_{k=0}^{n-1} \alpha_k {\bf b_k}$ and ${\bf x^{\alpha}}=\prod_{i=0}^{n-1}x_i^{\alpha_i}$. Note that this labeling depends on the choice of basis for $L_G$. This labeling of the vertices of the Scarf complex induces a labelling of the faces by labelling each face with the least common multiple of the labels of its vertices.  We denote this labelled simplicial complex by  ${\rm Scarf_{def}}(I_L)$. Given such a labelled simplicial complex we can associate a complex of free $\Bbbk[x_0,\dots,x_{n}]$-modules as described in \cite[Chapter 9.3]{MilStu05}.  For a labelled complex of $\Bbbk[x_0,\dots,x_{n}]$-modules $\mathcal{C}$ and for a vector ${\bf b}\in \mathbb{R}^{n+1}$, let $\mathcal{C}_{\leq {\bf b}}$ be the subcomplex of all faces of $\mathcal{C}$ such that the exponents of their labels are dominated coordinate-wise by ${\bf b}$.

 \begin{theorem} \label{Scarfdef_theo} The complex  of free  $\Bbbk[x_0,\dots,x_{n}]$-modules associated to ${\rm Scarf_{def}}(I_L)$  is exact and hence this complex is a free resolution of $I_L$.\end{theorem}
 
\begin{proof}

%%%A Deformation Sequence with the same initial ideal and lcm poset.

%%%Result of Peeva Sturmfels to conclude that the Scarf complex stabilizes.
  We use Theorem 5.4 of  \cite{PeevaStu98} that provides a description of the  Scarf complex of $I_L$ in terms of the Scarf complex of the initial ideal with respect to a degree reverse lexicographic order.  In particular, Theorem 5.4  of  \cite{PeevaStu98} shows that the abstract simplicial complex underlying the Scarf complex of $I_L$ depends only on the lcm poset.  Hence, the sequence of deformations $\tilde Q_i$ described above have associated lattice ideals with the same underlying Scarf complex.

%%%%%Using the Criterion of Exactness to deduce to deduce that the complex is exact.
By the criterion for exactness of a labelled complex described in \cite[Proposition 4.2]{BayStu98},  it suffices to show that the reduced homology group ${\tilde{H}}_j({\rm Scarf_{def}(I_L)}_{{\leq \bf b}}, \Bbbk)=0$ for every vector ${\bf b}$ in $\mathbb{Z}^{n+1}$ and for all integers $j \geq 0$.  Since the abstract simplicial complex underlying the Scarf complex of the lattice ideal is the same for sufficiently small deformations, we know that for every ${\bf b \in \mathbb{Z}}$, there exists  an $\epsilon>0$ such that  ${\tilde {H}}_j({ \rm Scarf_{def}}(I_L)_{ \leq {\bf b}}, \Bbbk)={\tilde{H}}_j(\rm Scarf(I_{L_{{\rm gen}})_{ \leq {\bf b_{\epsilon}}}}, \Bbbk))$ where ${\bf b_{\epsilon}}$ is the coordinate-wise maximum of all points ${\bf \alpha_{\epsilon}}$ in $L_{\rm gen}$ such that ${\bf  \alpha}$ is a point in $L$ that is dominated coordinate-wise by ${\bf b}$ (recall the construction of ${\bf \alpha_{\epsilon}}$ from ${\bf \alpha}$ described in the paragraph before Theorem \ref{Scarfdef_theo}).  From \cite[Proposition 4.2]{BayStu98},  ${\tilde{H}}_j (\rm Scarf(I_{L_{{\rm gen}})_{ \leq {\bf b_{\epsilon}}}}, \Bbbk)=0$ for all ${\bf b_{\epsilon}} \in \mathbb{R}^{n}$ and for all integers $j \geq 0$. Hence, the complex  ${\rm Scarf_{def}}(I_L)$  is exact. 

\end{proof}

A similar method is used by Manjunath and Sturmfels \cite{ManStu13} to construct a non-minimal free resolution for toppling ideals, i.e., the special case where the lattice is generated by the rows of the Laplacian of an undirected connected graph.  In the following example, we illustrate the construction of ${\rm Scarf_{def}}(I_L)$ and show that  ${\rm Scarf_{def}}(I_L)$ is not in general minimal.

\begin{example}\label{scarf_ex}
Consider the lattice ideal:
  $I_L=\langle x_2^{2}-x_1x_0, x_2^4-x_0^2x_1^2,x_2^6-x_0^3x_1^3, x_1^3-x_0x_2^2, x_1^2-x_0^2, x_1x_2^2-x_0^3,  x_2^4-x_0^4, x_1^6-x_0^2x_2^4, x_1^5-x_0^3x_2^2, 
 x_1^4-x_0^4,x_1^3x_2^2-x_0^5 \rangle$ from Example \ref{dirgraph_ex}. 
 
 Its deformation constructed in Example \ref{deform_ex} with $p=1,q=2$. 
 
$I_{L_{\rm gen}}=\langle x_2^{26}-x_0^{12}x_1^{14}, x_2^{52}-x_0^{24}x_1^{28},x_2^{78}-x_0^{36}x_1^{42},x_1^{39}-x_0^{12} x_2^{27},x_1^{25}-x_0^{24}x_2, x_1^{11}x_2^{25}-x_0^{36},x_2^{51}-x_0^{48}x_1^3,x_1^{78}-x_0^{24}x_2^{54},x_1^{64}-x_0^{36}x_2^{28}, x_1^{50}-x_0^{48} x_2^2, x_0^{60}-x_1^{36} x_2^{24}  \rangle$.

As describe in the algorithm, we construct the Scarf complex of the Laurent monomial module corresponding to $L_{\rm gen}$ and then relabel the vertices of this Scarf complex by Laurent monomials ${\bf x}^{\alpha}$ where $\alpha$ is a point in $L$. For example, the vertex corresponding to the point $(26,-12,-14)$ is labelled with the Laurent monomial $x_1^{2}x_2^{-1}x_3^{-1}$, the vertex corresponding to the point $(78,-36,-42)$ is labelled with $x_1^{6}x_2^{-3}x_3^{-3}$, the vertex corresponding to their sum $(104,-48,-56)$ is labelled with the product $x_1^{8}x_2^{-4}x_3^{-4}$ of their labels.

Using \cite[Theorem 5.2]{PeevaStu98},  we can determine the number of faces of a given dimension in its Scarf complex from its initial ideal with respect to any degree reverse lexicographic order. For the degree reverse lexicographic term order with the order $x_0>x_1>x_2$ on the variables,  the initial ideal is:

$M= \langle  x_1^{25}, x_0^{12} x_1^{14}, x_0^{36}\rangle$

Since $M$ is a monomial ideal in two variables, we conclude that its Scarf complex has dimension one. Furthermore, since $M$ has three minimal generators, its Scarf complex has three vertices and two edges. Hence, by \cite[Theorem 5.2]{PeevaStu98}, the Scarf complex of $I_{L_{\rm gen}}$ has one vertex, three edges and two faces and is a triangulation of the two-dimensional torus $\mathbb{R}^3/L_{\rm gen}$. We degenerate the Scarf complex of $I_{L_{\rm gen}}$ to construct a free resolution of the lattice ideal $\Bbbk[x_0,x_1,x_2]/I_L$ and hence, the rank of free modules of this free resolution is one, three and two in  homological degrees zero, one and two respectively. This free resolution is non-minimal since the Betti numbers of $\Bbbk[x_0,x_1,x_2]/I_L$ as computed using the package Macaulay 2 read  one, two and one respectively. 

\end{example}

{\bf Acknowledgments:} We thank Anders Jensen, Bernd Sturmfels and Josephine Yu for their helpful comments and discussions on this project. We are very grateful to Matt Baker and Frank-Olaf Schreyer for their guidance and support.  Thanks to Farbod Shokrieh for his encouragement.

\bigskip
\medskip

\begin{small}
{
\noindent
Madhusudan Manjunath, Department of Mathematics,\\
University of California, Berkeley,
USA. \\
{\tt madhu@math.berkeley.edu}

\medskip

\noindent
Spencer Backman, Department of Computer Science,\\
University of Rome ``La Sapienza", Rome,
Italy.\\
{\tt backman@di.uniroma1.it}
} 
\end{small}

\end{document}